\documentclass[10pt, a4paper]{amsart}
\usepackage{amssymb, amsmath, amsfonts, mathabx, mathrsfs, amsthm}
\usepackage[utf8]{inputenc}
\usepackage{lmodern}
\usepackage[T1]{fontenc}
\usepackage[a4paper,includeheadfoot,margin=2.54cm]{geometry}
\usepackage[matrix, arrow, curve]{xy}

\DeclareMathOperator{\Aut}{\mathrm{Aut}}

\numberwithin{equation}{section}

\begin{document} 

\title{Finite groups acting on compact complex parallelizable manifolds}
\date{}
\author{Aleksei Golota}

\newtheorem{theorem}{Theorem}[section] 
\newtheorem*{ttheorem}{Theorem}
\newtheorem{lemma}[theorem]{Lemma}
\newtheorem{proposition}[theorem]{Proposition}
\newtheorem*{conjecture}{Conjecture}
\newtheorem{corollary}[theorem]{Corollary}

{\theoremstyle{remark}
\newtheorem{remark}[theorem]{Remark}
\newtheorem{example}[theorem]{Example}
\newtheorem{notation}[theorem]{Notation}
\newtheorem{question}[theorem]{Question}
}

\theoremstyle{definition}
\newtheorem{construction}[theorem]{Construction}
\newtheorem{definition}[theorem]{Definition}

\begin{abstract} We prove that the automorphism group of a compact complex parallelizable manifold is Jordan. In the course of the proof we show that outer automorphism groups of cocompact lattices in complex Lie groups have bounded finite subgroups.
\end{abstract}

\maketitle

\section{Introduction} 

The study of finite subgroups in the automorphism groups of compact complex (e. g. projective algebraic) manifolds has attracted a lot of attention in recent years. These subgroups often satisfy various boundedness properties, for example, the following property, introduced in \cite{Pop11}.

\begin{definition} \label{Jordan} Let $G$ be a group. We say that $G$ is {\em Jordan} (or has the {\em Jordan property}) if there is a constant $J(G) \in \mathbb{N}$ such that for any finite subgroup $H \subset G$ there is a normal abelian subgroup $A \unlhd H$ of index at most $J(G)$. 
\end{definition}

A classical result of C. Jordan (\cite[Chapter 2]{Jor78}, see e. g. \cite{Rob90} for a modern proof) says that this property holds for $G = \mathrm{GL}_n(\mathbb{C})$. Clearly, a subgroup of a Jordan group is also Jordan; therefore, all linear algebraic groups over $\mathbb{C}$ are Jordan. 

Let $X$ is a projective variety over an algebraically closed field $k$ of zero characteristic. Then it is well known (see e.g. \cite{Gro61}) that the group $\Aut(X)$ of biregular automorphisms of $X$ is an algebraic group scheme locally of finite type over $k$. There is an exact sequence \begin{equation}\label{autseq} 1 \to \Aut^0(X) \to \Aut(X) \to \Aut(X)/\Aut^0(X) \to 1,
\end{equation}
where $\Aut^0(X)$ is the connected component of the identity and $\Aut(X)/\Aut^0(X)$ is a countable discrete group. S. Meng and D.-Q. Zhang \cite[Theorem 1.6]{MZ18} have shown that in this case the group $\Aut(X)$ is Jordan. J. Kim \cite[Theorem 1.1]{Kim18} has extended their result to the groups of biholomorphic automorphisms of compact K\"ahler spaces. In fact, for every compact complex manifold $X$ one can consider the exact sequence \eqref{autseq}. The connected component $\Aut^0(X)$ is known to be a complex Lie group (see e. g. \cite{Akh95}). If $X$ is compact Kaehler then it follows from the results in \cite{Fuj78, Lie78} that the action of $\Aut(X)$ on $H^2(X, \mathbb{Z})$ (modulo torsion) yields a homomorphism $$\Aut(X)/\Aut^0(X) \to \mathrm{GL}_N(\mathbb{Z})$$ with finite kernel; this is essential for the argument in \cite{Kim18} to show that $\Aut(X)$ is Jordan. In general, not much is known about the group of connected components of $\Aut(X)$ for an arbitrary compact complex manifold $X$.  

For non-K\"ahler compact complex manifolds there are only a few known results on the Jordan property for automorphism groups. \begin{itemize} \item S. Meng, F. Perroni and D.-Q. Zhang \cite[Corollary 1.2]{MPZ22} have proved this property for the automorphism groups of compact complex manifolds of Fujiki's class $\mathcal{C}$ (bimeromorphic to compact K\"ahler manifolds).
\item Yu. Prokhorov and C. Shramov in \cite[Theorem 1.6]{PS21} have shown that the automorphism groups of compact complex surfaces are Jordan. \item For the automorphism groups of Hopf manifolds the Jordan property was shown by A. Savelyeva in \cite[Theorem 1.4]{Sav20}. \item For examples of non-K\"ahler holomorphically symplectic manifolds constructed by D. Guan in \cite[Theorem 1]{Gu95a} and \cite[Theorem 3]{Gu95b}, the Jordan property of the automorphism groups was proved by F. Bogomolov, N. Kurnosov, A. Kuznetsova and E. Yasinsky in \cite[Theorem ~C]{BKKY22}. \item A recent result by K. Loginov \cite[Theorem 1.2]{Log22} says, in particular, that the automorphism group of any compact complex manifold $X$ of Kodaira dimension at least $\dim(X) - 2$ is Jordan. \end{itemize}

In this paper we consider finite subgroups in the automorphism groups of compact complex manifolds which are holomorphically parallelizable.

\begin{definition} A connected complex manifold $X$ is called parallelizable if its holomorphic tangent bundle is trivial.   
\end{definition}

Compact manifolds in this class are isomorphic to quotients of complex Lie groups by discrete cocompact subgroups (\cite[Theorem 1]{Wa54}). Below are a few well-known examples of compact complex parallelizable manifolds.

\begin{example} By \cite[Corollary 2]{Wa54}, a compact complex parallelizable manifold $X$ is K\"ahler if and only if it is biholomorphic to a compact complex torus $T = \mathbb{C}^n/\Gamma$, where $\Gamma \simeq \mathbb{Z}^{2n}$. In this case the group of automorphisms of $X$ has an explicit description, namely, it is isomorphic to a semidirect product $$T \rtimes \{A \in \mathrm{GL}_n(\mathbb{C}) \mid A(\Gamma) = \Gamma\},$$ see e. g. \cite[~p. 42]{Akh95}.
\end{example}

\begin{example} Let $G$ be the group of upper-triangular matrices of the form $$\begin{pmatrix} 1 & x & z \\ 0 & 1 & y\\ 0 & 0 & 1\end{pmatrix}, x, y, z \in \mathbb{C}.$$ Consider the subgroup $\Gamma \subset G$ of matrices with entries $x, y$ and $z$ in $\mathbb{Z}[\sqrt{-1}]$. Then the quotient $X = G/\Gamma$ is a compact complex non-K\"ahler manifold, called the Iwasawa manifold (see e. g. \cite[1.4]{Win98}).
\end{example}

\begin{example} Consider the complex solvable Lie group $G = \mathbb{C}^2 \rtimes_{\varphi} \mathbb{C}$, where the action $\varphi \colon \mathbb{C} \to \mathrm{GL}_2(\mathbb{C})$ is defined by $$\varphi(x + \sqrt{-1}y) = \begin{pmatrix} e^x & 0\\ 0 & e^{-x}
\end{pmatrix}.$$ There exists $a \in \mathbb{R}\setminus\{0\}$ such that the matrix $$\begin{pmatrix} e^a & 0\\ 0 & e^{-a}
\end{pmatrix}$$ is conjugate to a matrix in $\mathrm{SL}_2(\mathbb{Z})$ (for example, one can simply take $a = \log 2$). There exists a lattice $\Gamma' \subset \mathbb{C}^2$ such that for any $b \in \mathbb{R}\setminus\{0\}$ the subgroup $$\Gamma = \Gamma' \rtimes_{\varphi} \{a\mathbb{Z} + \sqrt{-1}b\mathbb{Z}\} \subset \mathbb{C}^2 \rtimes_{\varphi} \mathbb{C}$$ is a lattice in $G$. The quotient $X = G/\Gamma$ is a compact complex parallelizable manifold, first constructed by I. Nakamura in \cite{Na75}.
\end{example}

For more examples of compact complex parallelizable manifolds we refer to \cite{Win98} and references therein.

Our main result is the following theorem.

\begin{theorem} \label{main} Let $X$ be a compact complex parallelizable manifold. Then the group $\Aut(X)$ is Jordan.
\end{theorem}

Let us fix the notation. For a group $G$ we denote by $\Aut(G)$ the group of automorphisms of $G$, by $$\mathrm{Inn}(G) = \{h \mapsto ghg^{-1} \mid g \in G\} \subset \Aut(G)$$ the group of inner automorphisms of $G$ and by $\mathrm{Out}(G) = \Aut(G)/\mathrm{Inn}(G)$ the group of outer automorphisms of $G$. Also, if $G$ is a group and $H \subset G$ is a subgroup, we denote by $$\Aut(G; H) = \{\varphi \in \Aut(G) \mid \varphi(H) = H\}$$ the group of automorphisms of $G$ which map the subgroup $H \subset G$ isomorphically onto itself. The center of a group $G$ is denoted by $Z(G)$. If $H \subset G$ is a subgroup then $C_G(H)$ and $N_G(H)$ denote the centralizer of $H$ in $G$ and the normalizer of $H$ in $G$, respectively. If $G$ is a Lie group, then $\Aut(G)$ (respectively, $\mathrm{Out}(G)$) denotes the group of automorphisms of $G$ (respectively, outer automorphisms of $G$) as a Lie group.

In the course of the proof of Theorem \ref{main} we establish boundedness of finite subgroups (see Definition \ref{bfs}) in the groups $\mathrm{Out}(\Gamma)$ for cocompact lattices $\Gamma$ in complex Lie groups.

\begin{theorem}\label{main1} Let $\Gamma$ be a cocompact lattice in a connected complex Lie group $G$. Then the group $\mathrm{Out}(\Gamma)$ has bounded finite subgroups.
\end{theorem}

Let us outline the structure of the paper. In Section 2 we collect some preliminaries from group theory, including boundedness of finite subgroups and Jordan property. We also discuss group extensions and their automorphisms. In Section 3 we discuss group-theoretic properties of lattices in complex Lie groups. Section 4 is devoted to description of the structure of compact complex parallelizable manifolds, including their automorphism groups, following \cite{Win98}. Section 5 is expository; we present there a detailed proof of well-known finiteness of $\mathrm{Out}(\Gamma)$ for a cocompact lattice $\Gamma$ in a complex semisimple Lie group $G$. In Section 6 we discuss embeddings of lattices to Lie groups and some applications of rigidity to subgroups of lattices. Finally, in Section 7 we prove Theorem \ref{main1} and derive Theorem \ref{main} from this result. 

\textbf{Acknowledgement.} The author would like to thank his advisor Constantin Shramov for suggesting this problem and for his constant support. The author also thanks A. Klyachko and D. Tereshkin for useful conversations on group theory, Y. de Cornulier for answering his question on Mathoverflow and D. Timashev and the anonymous referee for helpful remarks. This work was performed at the Steklov International Mathematical Center and supported by the Ministry of Science and Higher Education of the Russian Federation (agreement no. 075-15-2022-265) and by the Basic Research Program of the National Research University Higher School of Economics.

\section{Preliminaries}

We collect here a few definitions and facts on Jordan groups, their quotients and extensions. First, we quote a generalization of the Jordan theorem to connected Lie groups from \cite[Theorem 1]{BW65} (see also \cite[Theorem 2]{Pop18}).

\begin{theorem}\label{liejor} Let $G$ be a connected real Lie group. Then it is Jordan.
\end{theorem}

Note that the quotient groups and extensions of Jordan groups are not necessarily Jordan. To overcome this difficulty it is helpful to consider a more restrictive property of groups (\cite[Definition 2.7]{Pop11}).

\begin{definition} \label{bfs} We say that a group $G$ has {\em bounded finite subgroups} if there is a number $B(G) \in \mathbb{N}$ such that for any finite subgroup $H \subset G$ we have $|H| \leqslant B(G)$. 
\end{definition}

\begin{example} \label{abelianbfs} If $G$ is a finitely generated abelian group then by the structure theorem it is isomorphic to a direct sum of a free abelian group and the (finite) torsion subgroup. In particular, $G$ has bounded finite subgroups.
\end{example}

\begin{remark} There exist finitely generated groups with unbounded finite subgroups. For a simple example, let $G$ be the group of permutations of $\mathbb{Z}$ which are shifts up to permutations of finite subsets of $\mathbb{Z}$. Then $G$ is generated by a shift $n \mapsto n+1$ and a transposition interchanging 0 and 1. Moreover, $G$ contains $S_n$ as a subgroup for every $n \in \mathbb{N}$, so it has unbounded finite subgroups. Since the center of $G$ is trivial, $G$ embeds to $\Aut(G)$, which therefore has unbounded finite subgroups. For an example of a finitely presented group with unbounded finite subgroups one may take e. g. Thompson's group $V$.
\end{remark}

\begin{remark}\label{quotbfs} In general, the property of having bounded finite subgroups is not preserved under taking quotients. However, the quotient of a group with bounded finite subgroups by a finite normal subgroup has bounded finite subgroups. 
\end{remark}

Another simple but useful observation is the following.

\begin{lemma}\label{bfsindex} Let $G$ be a group and let $F \subset G$ be a subgroup of finite index. Then $F$ has bounded finite subgroups if and only if $G$ does.
\end{lemma}

\begin{proof} Let $H \subset G$ be a finite subgroup. Then the intersection $F \cap H$ is a finite subgroup of $F$, so by assumption there is a constant $B$ such that $|F \cap H| \leqslant B$. Then we have $$|H| \leqslant |H \cap F|\cdot[G: F] = B\cdot[G: F],$$ so $G$ also has bounded finite subgroups. 
\end{proof}

More nontrivial examples of groups with bounded finite subgroups are given by a classical theorem of H. Minkowski (see e.g. \cite[Theorem 1]{Ser07}). 

\begin{theorem}\label{Minkowski} The orders of finite subgroups in $G = \mathrm{GL}_n(\mathbb{Z})$ are bounded by a natural number $M(n)$ depending on $n$ only.
\end{theorem}

An easy but important result below (see \cite[Lemma 2.9]{Pop11}) says that an extension of a group with bounded finite subgroups by a Jordan group is Jordan as well.

\begin{proposition}\label{Extensions} Consider an exact sequence of groups $$ 1 \to G_1 \to G_2 \to G_3 \to 1.$$ Suppose that $G_1$ is Jordan and $G_3$ has bounded finite subgroups. Then $G_2$ is Jordan.
\end{proposition}

Suppose that $G$ is a Jordan group and let $H \unlhd G$ be a normal subgroup. One may ask if the quotient group $G/H$ is Jordan provided that the subgroup $H$ is ``nice'' in a certain sense. This is clearly the case if $H$ is finite (see \cite[Lemma 2.5]{Pop11}). The example below shows that in general it is not the case, even if $H$ is infinite cyclic.

\begin{example}
Here is an example of a Jordan group $G$ containing a normal subgroup $H \simeq \mathbb{Z}$ such that the quotient group $G' = G/H$ is not Jordan. Let $p > 2$ be a prime number and consider the semidirect product $$G_p = (\mathbb{Z}/p\mathbb{Z})^p \rtimes_{\varphi} \mathbb{Z},$$ where the generator of $\mathbb{Z}$ is mapped by the homomorphism $\varphi \colon \mathbb{Z} \to \mathrm{Aut}((\mathbb{Z}/p\mathbb{Z})^p)$ to a cyclic permutation of coordinates in $(\mathbb{Z}/p\mathbb{Z})^p$, viewed as a vector space over $\mathbb{Z}/p\mathbb{Z}$. Then the subgroup $$H_p = p\mathbb{Z} \subset \mathbb{Z}$$ acts trivially on $(\mathbb{Z}/p\mathbb{Z})^p$. Let $i_p \colon \mathbb{Z} \to G_p$ be the inclusion of $H_p$. We construct the group $$G = \Asterisk_{\mathbb{Z}}G_p$$ as the free product of $G_p$ for all $p > 2$ with amalgamated subgroups $H_p \simeq \mathbb{Z}$ embedded via $i_p \colon \mathbb{Z} \to G_p$. Denote by $H \subset G$ the image of all subgroups $H_p$; it is a normal infinite cyclic subgroup of $G$. Note that every finite subgroup of $G_p$ is abelian, since it is contained in the kernel of the quotient map $$G_p \to G_p/(\mathbb{Z}/p\mathbb{Z})^p \simeq \mathbb{Z}.$$ From a generalization of \cite[Theorem I.4.8]{Ser80} to graphs of groups it follows that every finite subgroup of $G$ is conjugate to a finite subgroup of one of $G_p$; therefore, it is also abelian. Hence, the group $G$ is Jordan with $J(G) = 1$. On the other hand, the quotient $G' = G/H$ is isomorphic to the free product of quotient groups $$G'_p = G_p/H_p \simeq (\mathbb{Z}/p\mathbb{Z})^p \rtimes_{\varphi} \mathbb{Z}/p\mathbb{Z}.$$ Therefore, for every $p > 2$ there exists a subgroup $G'_p \subset G'$ such that the minimal index of a normal abelian subgroup $A \unlhd G'_p$ is at least $p$. This means that $G'$ is not Jordan.
\end{example}

\begin{remark} In the above example the groups $G$ and $G'$ are not finitely generated. It would be interesting to construct counterexamples of this kind with $G$ finitely generated and $H \simeq \mathbb{Z}$, or to prove that they do not exist.   
\end{remark}

Let $F$ be a group and let $$1 \to K \to F \to Q \to 1$$ be an exact sequence of groups. Also let $\psi \colon Q \to \mathrm{Out}(K)$ the corresponding homomorphism. In \cite{Wel71} C. Wells initiated a study of the group $\Aut(F; K)$ by cohomological methods. He constructed an exact sequence $$1 \to Z_{\psi}^1(Q, Z(K)) \to \Aut(F; K) \to \Aut(F) \times \Aut(K),$$ where $Z_{\psi}^1(Q, Z(K))$ is the group of 1-cocycles.

In \cite{Mal02} W. Malfait constructed several exact sequences for the group $$\mathrm{Out}(F; K) = \Aut(F; K)/\mathrm{Inn}(F)$$ elaborating on the construction of Wells (see \cite[Theorem 4.10]{Mal02}). We state here (partial versions of) some of his results, which we use later. The theorem below is used in the proof of Theorem \ref{outgamma}; for the proof see \cite[Theorem 3.10]{Mal02}.

\begin{theorem}\label{malfait11} Consider an exact sequence of groups
\begin{equation*}
\xymatrix{
1 \ar[r] & K \ar[r] & F \ar[r]^{j} & Q \ar[r] & 1.\\
}
\end{equation*}
Denote by $\psi \colon Q \to \mathrm{Out}(K)$ the corresponding natural homomorphism. Let us denote by $${\overline B}^1_{\psi}(Q, Z(K)) = \frac{C_F(K) \cap j^{-1}(Z(Q))}{Z(F)}$$ the subgroup of inner automorphisms of $F$ inducing the identity on both $K$ and $Q$, and by $$ \overline{H}_{\psi}^1(Q, Z(K)) = Z^1_{\psi}(Q, Z(K))/\overline{B}^1_{\psi}(Q, Z(K))$$ the quotient group. Then there exists an exact sequence of groups
\begin{equation*} 1 \to {\overline H}^1(Q, Z(K)) \to \mathrm{Out}(F; K) \to \frac{\Aut(K) \times \Aut(Q)}{F/(C_F(K) \cap j^{-1}(Z(Q)))}.
\end{equation*}
Moreover, the group $\overline{H}^1_{\psi}(Q, Z(K))$ is a quotient of the first cohomology group $H^1_{\psi}(Q, Z(K))$.
\end{theorem}

The next theorem is used in the proof of Theorem \ref{main1}; for the proof see \cite[Theorem 4.8]{Mal02}.

\begin{theorem} \label{malfait} Let $1 \to K \to F \to Q \to 1$ be an exact sequence of groups. Denote by $\psi \colon Q \to \mathrm{Out}(K)$ the corresponding natural homomorphism. Consider the homomorphism $$B \colon \Aut(F; K) \to \Aut(Q).$$ Then there exists an exact sequence of groups 
\begin{equation}\label{malfait1} 1 \to \Lambda_B \to \mathrm{Out}(F; K) \to \frac{\mathrm{Im}(B)}{\mathrm{Inn}(Q)} \to 1.
\end{equation}
In the above sequence the quotient group $\mathrm{Im}(B)/\mathrm{Inn}(Q)$ embeds to $\mathrm{Out}(Q)$. Moreover, the group $\Lambda_B$ can be expressed via an exact sequence 
\begin{equation}\label{malfait2} 1 \to \overline{H}_{\psi}^1(Q, Z(K)) \to \Lambda_B \to \Xi \to 1,
\end{equation}
where $\Xi$ is a subgroup of $\mathrm{Out}(K)/\psi(Z(Q))$.
\end{theorem}

\section{Discrete subgroups in Lie groups}

In this section we discuss a few group-theoretical properties of lattices in complex Lie groups. Recall that a lattice in a connected Lie group $G$ is a discrete subgroup $\Gamma \subset G$ such that the quotient space $G/\Gamma$ has a $G$-invariant measure with finite volume. A lattice $\Gamma \subset G$ is called cocompact if the quotient $G/\Gamma$ is compact. First of all, we quote the following result from \cite[Theorem 6.15]{Rag72}.

\begin{theorem}\label{fingen} Let $\Gamma$ be a cocompact lattice in a connected Lie group $G$. Then $\Gamma$ is finitely presented.
\end{theorem}

Below are some basic definitions and results on lattices in connected solvable Lie groups. 

\begin{definition} A group $\Gamma$ is called polycyclic if there exists a sequence of subgroups $$\Gamma = \Gamma_0 \supset \Gamma_1 \supset \cdots \supset \Gamma_n = \{id\},$$ such that for all $0 \leqslant i \leqslant n-1$ the subgroup $\Gamma_{i+1}$ is normal in $\Gamma_i$ and the quotient $\Gamma_i/\Gamma_{i+1}$ is a cyclic group.
\end{definition}

Equivalently, polycyclic groups can be defined as solvable groups such that every subgroup is finitely generated \cite[Theorem 1.3]{Hir38}. The class of polycyclic groups is closed under passing to subgroups and quotient groups \cite[Theorem 1.2]{Hir38}.

\begin{proposition} \label{subquot} Subgroups and quotient groups of polycyclic groups are also polycyclic.
\end{proposition}

The following proposition (see \cite[Proposition 3.7]{Rag72}) says that lattices in simply connected solvable Lie groups are polycyclic with infinite subquotients.

\begin{proposition} \label{latsolv} Let $G$ be a simply connected solvable Lie group and $\Gamma$ a lattice in $G$. Then there exists a sequence of subgroups $$\Gamma = \Gamma_0 \supset \Gamma_1 \supset \cdots \supset \Gamma_n = \{id\},$$ such that for all $0 \leqslant i \leqslant n-1$ the subgroup $\Gamma_{i+1}$ is normal in $\Gamma_i$ and the quotient $\Gamma_i/\Gamma_{i+1}$ is infinite cyclic. Moreover, the number $n$ is equal to the real dimension of $G$.
\end{proposition}

The group of outer automorphisms of a polycyclic group is linear over $\mathbb{Z}$ (see \cite[Corollary on p. 543]{We94}).

\begin{theorem} \label{autopoly} Let $\Gamma$ be a polycyclic group. Then there exists $r \in \mathbb{N}$ such that the group $\mathrm{Out}(\Gamma)$ of outer automorphisms of $\Gamma$ is isomorphic to a subgroup of $\mathrm{GL}_r(\mathbb{Z})$.
\end{theorem}

\begin{remark}\label{baues} O. Baues and F. Grunewald proved that for any polycyclic-by-finite group $\Gamma$ the group $\mathrm{Out}(\Gamma)$ is an arithmetic group \cite[Theorem 1.1]{BG06}.
\end{remark}

We obtain the following corollary.

\begin{corollary}\label{outerpoly} Let $\Gamma$ be a lattice in a connected solvable Lie group $G$. Then the group $\mathrm{Out}(\Gamma)$ has bounded finite subgroups.
\end{corollary}

\begin{proof} Consider the universal cover $p \colon \widetilde{G} \to G$. Then $\widetilde{G}$ is also solvable and the preimage $\widetilde{\Gamma} = p^{-1}(\Gamma)$ is a lattice in $\widetilde{G}$, so by Proposition \ref{latsolv} the group $\widetilde{\Gamma}$ it is polycyclic. Since $\Gamma$ is a quotient of $\widetilde{\Gamma}$ by $\mathrm{Ker}(p) \cap \widetilde{\Gamma}$, it is also polycyclic by Proposition \ref{subquot}. Thus the group $\mathrm{Out}(\Gamma)$ has an embedding to $\mathrm{GL}_r(\mathbb{Z})$ for some number $r \in \mathbb{N}$ by Theorem \ref{autopoly}. So by Theorem \ref{Minkowski} the group $\mathrm{Out}(\Gamma)$ has bounded finite subgroups.
\end{proof}

We recall an important result on compatibility of lattices in complex Lie groups with the Levi--Mal'cev decomposition (for the proof, see \cite[~Theorem 3.5.3]{Win98} and \cite[~Theorems 4.3, 4.5, 4.7]{VGS88}). 

\begin{proposition} \label{hered} Let $G$ be a connected complex Lie group and let $$1 \to R \to G \to S \to 1$$ be its Levi--Mal'cev decomposition. Let $\Gamma \subset G$ be a cocompact lattice. Then the group $\Gamma \cap R$ is a cocompact lattice in $R$ and the quotient group $\Gamma/(\Gamma \cap R)$ is a cocompact lattice in $S$. 
\end{proposition}

\begin{remark} Note that Proposition \ref{hered} does not hold in general for real Lie groups, see \cite{Ge15} and references therein.
\end{remark}

The following theorem is a consequence of \cite[Theorem XVIII.4.6]{Hoch65} and \cite[Theorem 1.4.2]{VGO90}.

\begin{theorem}\label{linear} Let $G$ be a simply connected complex Lie group. Then there exists an injective homomorphism of complex Lie groups $$\rho \colon G \to \mathrm{GL}_n(\mathbb{C})$$ for some $n \in \mathbb{N}$.
\end{theorem}

We will need a result on density of lattices in finite-dimensional representations of complex Lie groups (see \cite[Theorem 3.1]{Mos78} and \cite[Theorem 3.4.1]{Win98}). It generalizes a well-known density theorem for lattices in semisimple groups, due to A. Borel.

\begin{theorem}\label{density} Let $G$ be a connected complex Lie group and let $\rho \colon G \to \mathrm{GL}_n(\mathbb{C})$ be a representation. Suppose that $\Gamma \subset G$ is a lattice. Then the closures of $\rho(\Gamma)$ and $\rho(G)$ in Zariski topology coincide with each other.   
\end{theorem}

We will also need the following result about centers of lattices in linear complex Lie groups.

\begin{proposition} \label{zsemis} Let $G$ be a connected complex Lie group which admits a faithful linear representation $$\rho \colon G \to \mathrm{GL}_n(\mathbb{C})$$ with Zariski-closed image and let $\Gamma \subset G$ be a lattice. Then there is an equality $$Z(\Gamma) = Z(G) \cap \Gamma.$$ Moreover, if $G$ is semisimple then both $Z(G)$ and $Z(\Gamma)$ are finite.
\end{proposition}

\begin{proof} Consider a faithful linear representation $\rho \colon G \to \mathrm{GL}_n(\mathbb{C})$ with Zariski-closed image. Then we can identify $G$ with $\rho(G)$ and consider $G$ as an algebraic subgroup of $\mathrm{GL}_n(\mathbb{C})$. Since conjugation is an algebraic operation on $G$, we have $$C_{G}(H) = C_{G}(\overline{H})$$ for any subgroup $H \subset G$. By Theorem \ref{density} the subgroup $\Gamma$ is dense in $G$ in Zariski topology. Therefore, we have $$C_{G}(\Gamma) = C_{G}(\overline{\Gamma}) = C_G(G)= Z(G).$$ Intersecting with $\Gamma$, we obtain $$Z(\Gamma) = C_G(\Gamma) \cap \Gamma = Z(G) \cap \Gamma,$$ as desired. 

Suppose now that $G$ is semisimple. Since $Z(G) \subset G$ is a closed algebraic subgroup and $Z(G)^0 = \{e\}$, it is finite. Thus $Z(\Gamma)$ is also finite in this case. 
\end{proof}

\begin{remark}\label{linearity} In Proposition \ref{zsemis} we assume the complex Lie group $G$ to be linear, that is, to admit a faithful linear representation $$\rho \colon G \to \mathrm{GL}_n(\mathbb{C})$$ with image being closed in Zariski topology. This is always the case for complex semisimple Lie groups (see \cite[Theorem 6.3]{VGO90}). Note that by Theorem \ref{linear} and \cite[Lemma 1.11.2]{Win98} simply connected complex Lie groups are linear as complex analytic groups (not necessarily algebraic). The conclusion of Proposition \ref{zsemis} is also true in this case, e.g. by the argument in the proof of Proposition \ref{outer} below.
\end{remark}

\section{Compact parallelizable manifolds}

Recall the structure theorem for compact complex parallelizable manifolds, due to H.-C. Wang \cite[Theorem 1]{Wa54}.

\begin{theorem} \label{wang} Let $X$ be a compact complex parallelizable manifold. Then there exists a connected complex Lie group $G$ and a discrete cocompact subgroup $\Gamma \subset G$ such that $X$ is biholomorphic to the quotient $G/\Gamma$.
\end{theorem}

\begin{remark}\label{univcover} Let $G$ be a (complex) Lie group and let $$p \colon \widetilde{G} \to G$$ be its universal cover. Then $\widetilde{G}$ has a structure of a (complex) Lie group such that $p$ is a homomorphism of (complex) Lie groups. Let $\Gamma \subset G$ be a lattice and let $X = G/\Gamma$ be the quotient. Then $p^{-1}(\Gamma)$ is a lattice in $\widetilde{G}$ and $X$ is isomorphic to the quotient $\widetilde{G}/p^{-1}(\Gamma)$. Therefore when considering a compact complex parallelizable manifold $X$ we can always choose the Lie group $G$ to be simply connected.
\end{remark}

We denote by $\Aut(X)$ the automorphism group of $X$ and by $\Aut^0(X)$ the connected component of the identity. The quotient group $\Aut(X)/\Aut^0(X)$ is called the group of connected components of $\Aut(X)$. The next theorem of J. Winkelmann gives an explicit description of the automorphism group of a compact parallelizable manifold.

\begin{theorem} \label{auto} Let $G$ be a simply connected complex Lie group. Let $\Gamma \subset G$ be a discrete subgroup and let $X = G/\Gamma$. Assume that $H^0(X, \mathcal{O}_{X}) = \mathbb{C}$. Consider a semidirect product $$G \rtimes \Aut(G; \Gamma)$$ given by the natural action of $\Aut(G; \Gamma)$ on $G$. Then the following assertions hold. \begin{enumerate} \item There exists an isomorphism $$\Aut(X) \simeq \frac{G \rtimes \Aut(G; \Gamma)}{\tau(\Gamma)}.$$ Here the embedding $\tau \colon \Gamma \to G \rtimes \Aut(G; \Gamma)$ is defined by $\gamma \mapsto (\gamma, \mathrm{int}_{\gamma})$, where $\mathrm{int}_{\gamma} \in \Aut(G; \Gamma)$ is the conjugation by $\gamma$. \item The connected component of the identity $\Aut^0(X)$ is isomorphic to $G/(\Gamma \cap Z(G))$. \item The group of connected components $\Aut(X)/\Aut^0(X)$ is isomorphic to $\Aut(G; \Gamma)/\mathrm{Inn}(\Gamma)$. \end{enumerate}
\end{theorem}

\begin{proof} Assertions (1) and (2) are dealt with in \cite[Corollary 5.2.8]{Win98} and \cite[Corollary 5.2.9]{Win98}, respectively. For completeness, we show assertion (3). Consider the embedding of $G$ to $G \rtimes \Aut(G; \Gamma)$. Since $\Aut(G; \Gamma)$ is discrete, the image of $G$ is the connected component of the identity in $G \rtimes \Aut(G; \Gamma)$. By assertion (1) the automorphism group $\Aut(X)$ is isomorphic to the quotient $G \rtimes \Aut(G; \Gamma)/\tau(\Gamma)$. Also, by assertion ~(2) the connected component $\Aut^0(X)$ is isomorphic to $G/(\Gamma \cap Z(G))$. Therefore the quotient group $$\Aut(X)/\Aut^0(X)$$ is isomorphic to the quotient of $\Aut(G; \Gamma)$ by $\tau(\Gamma)/\tau(\Gamma \cap Z(G))$. By Proposition \ref{zsemis} the latter quotient is isomorphic to $\Gamma/Z(\Gamma) = \mathrm{Inn}(\Gamma)$. Therefore we obtain $\Aut(X)/\Aut^0(X) \simeq \Aut(G; \Gamma)/\mathrm{Inn}(\Gamma)$, as desired. 
\end{proof}

We also recall the proof of discreteness of the group $\Aut(G; \Gamma)$ (see \cite[Lemma 5.3.1]{Win98}).

\begin{proposition} \label{outer} Let $G$ be a simply connected complex Lie group. Let $\Gamma \subset G$ be a discrete subgroup and let $X = G/\Gamma$. Assume that $H^0(X, \mathcal{O}_X) = \mathbb{C}$. Then the group $\Aut(G; \Gamma)$ embeds into $\Aut(\Gamma)$. Moreover, the group of connected components $\Aut(X)/\Aut^0(X)$ embeds into $\mathrm{Out}(\Gamma)$.
\end{proposition}

\begin{proof} The kernel of the homomorphism $$\Aut(G; \Gamma) \to \Aut(\Gamma)$$ consists of all $\varphi \in \Aut(G; \Gamma)$ such that $\varphi|_{\Gamma} = \mathrm{Id}_{\Gamma}$. Consider the holomorphic map $\zeta \colon G \to G$ defined by $$\zeta(g) = \varphi(g)g^{-1}.$$ The condition $\varphi|_{\Gamma} = \mathrm{Id}_{\Gamma}$ implies that $\zeta(g\gamma) = \zeta(g)$ for every $g \in G$ and $\gamma \in \Gamma$. So the map $\zeta$ induces a holomorphic map from $X$ to $G$. Since the complex Lie group $G$ is simply connected, it is linear by Theorem \ref{linear} and therefore Stein as a complex manifold (see \cite[Corollary 1.11.3]{Win98}). Therefore the assumption $$H^0(X, \mathcal{O}_X) = \mathbb{C}$$ implies that the image of $\zeta$ is equal to $\{e\}$, so $\varphi = \mathrm{Id}_G$. Hence the group $\Aut(G; \Gamma)$ embeds to $\Aut(\Gamma)$. By assertion (3) of Theorem \ref{auto} the group $\Aut(X)/\Aut^0(X)$ is isomorphic to $\Aut(G; \Gamma)/\mathrm{Inn}(\Gamma)$ which then embeds to $\Aut(\Gamma)/\mathrm{Inn}(\Gamma) = \mathrm{Out}(\Gamma)$. 
\end{proof}

We recall a well-known fact about the group of outer automorphisms of a semisimple Lie group (see e.g. \cite[Theorem ~3.3.1]{VGO90}).

\begin{theorem}\label{semisimple} Let $G$ be a semisimple Lie group. Then the group $\mathrm{Out}(G)$ is finite.
\end{theorem}

J. Winkelmann obtained the following proposition as a corollary of Theorem \ref{semisimple} and assertion (3) of Theorem \ref{auto} (see \cite[Theorem 5.3.5]{Win98}). We do not use this result later on. Instead, in Sections 5 -- 7 below we study in detail the groups $\mathrm{Out}(\Gamma)$ for cocompact lattices $\Gamma$ in arbitrary complex Lie groups.

\begin{proposition} \label{outsemi} Let $G$ be a complex semisimple Lie group. Let $\Gamma \subset G$ a cocompact lattice and denote $X = G/\Gamma$. Then the group $\Aut(X)/\Aut^0(X)$ is finite.
\end{proposition}

\begin{remark} A. Borel proved in \cite{Bor63} that every complex semisimple Lie group admits a cocompact lattice. So there exist (in fact, only countably many up to isomorphism \cite[Proposition 3.13.2]{Win98}) compact complex parallelizable manifolds $X \simeq G/\Gamma$ with $G$ semisimple.
\end{remark}

\section{Outer automorphisms of lattices in semisimple Lie groups} 

In this section we provide a complete proof of the following statement. 

\begin{theorem}\label{outgamma} Let $\Gamma$ be a cocompact lattice in a connected complex semisimple Lie group $G$. Then the group $\mathrm{Out}(\Gamma)$ is finite.
\end{theorem}

This theorem is well-known, see e.g. \cite[Corollary (8.5)]{Wa72} for a proof in the case when $G$ has no factors isomorphic to $\mathrm{SL}_2(\mathbb{C})$. In general, Theorem \ref{outgamma} follows from the celebrated rigidity theorem of Mostow, as mentioned e. g. in \cite{Cor07}. However, since we consider lattices in semisimple Lie groups with nontrivial center, we have chosen to provide a detailed argument here.

First, we state the rigidity theorem of Mostow \cite[Theorem 24.1]{Mos73}.

\begin{theorem} \label{mostow} Let $G$ and $G'$ be connected real semisimple Lie groups having trivial center and no compact or $\mathrm{PSL}_2(\mathbb{R})$-factors. Let $\Gamma \subset G$ and $\Gamma' \subset G'$ be discrete cocompact subgroups. Then for every isomorphism of groups $$\theta \colon \Gamma \to \Gamma'$$ there exists an isomorphism of Lie groups $$\overline{\theta} \colon G \to G'$$ such that $\theta$ is the restriction of $\overline{\theta}$ to $\Gamma$.
\end{theorem}

We use Theorem \ref{mostow} to show that Theorem \ref{outgamma} holds assuming that $G$ has trivial center. 

\begin{proposition}\label{centerless} Let $G$ be a connected complex semisimple Lie group with trivial center and let $\Gamma$ be a cocompact lattice in $G$. Then the group $\mathrm{Out}(\Gamma)$ is finite.
\end{proposition}

\begin{proof} Let $\theta \colon \Gamma \to \Gamma$ be an automorphism. Since $G$ is a complex semisimple Lie group, Theorem \ref{mostow} applies to $\Gamma$ and gives an automorphism $\overline{\theta} \colon G \to G$ extending $\theta$. Moreover, by Remark \ref{linearity} there exists a faithful linear representation $\rho \colon G \to \mathrm{GL}_n(\mathbb{C})$ and $\rho(\Gamma)$ is dense in $\rho(G)$ by Theorem \ref{density}. Thus the automorphism $\overline{\theta}$ is uniquely determined by $\theta$. Therefore we obtain a homomorphism $\Aut(\Gamma) \to \Aut(G)$ sending $\theta$ to $\overline{\theta}$. The subgroup $\mathrm{Inn}(\Gamma) \subset \Aut(\Gamma)$ maps to $\mathrm{Inn}(G)$, therefore we obtain a homomorphism $$\mathrm{Out}(\Gamma) \to \mathrm{Out}(G),$$ whose kernel consists of outer automorphisms of $\Gamma$ extending to inner automorphisms of $G$. In other words, the kernel of the above homomorphism lies in the quotient $N_G(\Gamma)/\Gamma$ of the normalizer $N_{G}(\Gamma)$ by $\Gamma$. It is a closed subgroup of ~$G$; by discreteness of $\Gamma$ the connected component $N^{0}_{G}(\Gamma)$ centralizes $\Gamma$. Then Theorem \ref{density} implies that $$N^{0}_{G}(\Gamma) \subset C_G(\Gamma) = C_G(\overline{\Gamma}) = C_G(G) = Z(G),$$ and $Z(G)$ is trivial by assumption. Therefore $N_G(\Gamma)$ is a discrete subgroup of $G$. Since $\Gamma \subset G$ is cocompact, the quotient $N_G(\Gamma)/\Gamma$ is finite. The group $\mathrm{Out}(G)$ is also finite by Theorem \ref{semisimple}, so $\mathrm{Out}(\Gamma)$ is finite as well.
\end{proof}

Now we derive Theorem \ref{outgamma} from Proposition \ref{centerless}.

\begin{proof}[Proof of Theorem \ref{outgamma}] Consider the center $Z(G)$ of $G$. Since the Lie group $G$ is semisimple, it is linear (see Remark \ref{linearity}), so $Z(G)$ is finite by Proposition \ref{zsemis}. Moreover, $Z(G)$ is equal to the intersection of all maximal tori in $G$ (see, for example, \cite[Theorem 3.2.8]{VGO90}). Let us denote $G' = G/Z(G)$. Then ~$G'$ is also semisimple and the center of $G'$ is trivial. Consider the subgroup $$\Gamma' = \Gamma/(Z(G) \cap \Gamma) \subset G'$$ Then $\Gamma'$ is a cocompact lattice in $G'$. By Proposition \ref{zsemis} the center of $\Gamma'$ is also trivial, and we have the equality $Z(G) \cap \Gamma = Z(\Gamma)$. Now consider the exact sequence $$1 \to Z(\Gamma) \to \Gamma \to \Gamma' \to 1.$$ Since $Z(\Gamma) \subset \Gamma$ is a characteristic subgroup, one has $\mathrm{Out}(\Gamma; Z(\Gamma)) = \mathrm{Out}(\Gamma)$. We apply Theorem \ref{malfait11} and obtain an exact sequence 
\begin{equation}\label{exactseq}
1 \to H^1(\Gamma', Z(\Gamma)) \to \mathrm{Out}(\Gamma) \to \mathrm{Out}(\Gamma') \times \Aut(Z(\Gamma)).
\end{equation}
The natural homomorphism $\psi \colon \Gamma' \to \Aut(Z(\Gamma))$ is trivial, therefore we have $$H^1(\Gamma', Z(\Gamma)) = \mathrm{Hom}(\Gamma', Z(\Gamma)),$$ and the latter group is finite, since $Z(\Gamma)$ is finite by Proposition \ref{zsemis} and $\Gamma'$ is finitely generated by Theorem \ref{fingen}. The group $\mathrm{Out}(\Gamma')$ is finite by Proposition \ref{centerless}, and $\Aut(Z(\Gamma))$ is finite since $Z(\Gamma)$ is finite. Therefore from the exact sequence \eqref{exactseq} we obtain that $\mathrm{Out}(\Gamma)$ is also finite.
\end{proof}

\section{Application of rigidity to deformations of lattices}

In this section we recall a few facts about the space of embeddings of a cocompact lattice $\Gamma$ to a connected Lie group ~$G$, following \cite{Wei60, Wei62}, see also \cite[I.6]{VGS88}.

\begin{definition} Let $G$ be a connected Lie group and let $\Gamma \subset G$ be a cocompact lattice. We denote by $$\mathrm{Hom}(\Gamma, G)$$ the set of homomorphisms from $\Gamma$ to $G$ with the topology of pointwise convergence. Denote by $\mathcal{R}(\Gamma, G)$ the subset (with induced topology) of those $r \in \mathrm{Hom}(\Gamma, G)$ such that $r$ is injective and $r(\Gamma)$ is a cocompact lattice in $G$.
\end{definition}

The next proposition \cite[p. 370]{Wei60} describes the structure of $\mathrm{Hom}(\Gamma, G)$ and $\mathcal{R}(\Gamma, G)$ as topological spaces.

\begin{proposition}\label{analytic} Let $G$ be a Lie group and let $\Gamma \subset G$ be a cocompact lattice. Then the topological space $\mathrm{Hom}(\Gamma, G)$ is homeomorphic to a closed real-analytic subset of the $s$-fold product $G^s$. Moreover, the subspace $\mathcal{R}(\Gamma, G)$ is open in $\mathrm{Hom}(\Gamma, G)$ in Euclidean topology.
\end{proposition}

The connected component $\mathcal{R}_0(\Gamma, G) \subset \mathcal{R}(\Gamma, G)$ containing a fixed embedding $i \colon \Gamma \to G$ is called the space of deformations of $\Gamma$ in $G$. For lattices in complex semisimple Lie groups one has the following description of $\mathcal{R}_0(\Gamma, G)$, due to Weil \cite[Theorem 1]{Wei62}.

\begin{theorem}\label{weil} Let $G$ be a complex semisimple Lie group and let $\Gamma \subset G$ be a cocompact lattice. Then the connected component $\mathcal{R}_0(\Gamma, G)$ consists of embeddings obtained from $i$ by inner automorphisms of $G$.
\end{theorem}

Moreover, an important application of rigidity of cocompact lattices in complex semisimple Lie groups is the following result (see \cite[Theorem 7.63]{Rag72} and \cite[Theorem (8.4)]{Wa72}).

\begin{theorem}\label{finconn} Let $G$ be a complex semisimple Lie group and let $\Gamma \subset G$ be a cocompact lattice. Then the topological space $\mathcal{R}(\Gamma, G)$ has only finitely many connected components. 
\end{theorem}

Now we apply Theorems \ref{weil} and \ref{finconn} to show that the intersection $\Gamma \cap R$ is invariant by a subgroup of finite index in $\mathrm{Out}(\Gamma)$.

\begin{proposition}\label{finindex} Let $G$ be a connected complex Lie group and let $\Gamma \subset G$ be a cocompact lattice. Then the subgroup $$\mathrm{Out}(\Gamma; \Gamma \cap R) \subset \mathrm{Out}(\Gamma)$$ is of finite index.
\end{proposition}

\begin{proof} Let us fix an embedding $i \colon \Gamma \to G$ and consider the space $\mathcal{R}(\Gamma, G)$. For every $\theta \in \Aut(\Gamma)$ the composition $i \circ \theta \colon \Gamma \to G$ is an element of $\mathcal{R}(\Gamma, G)$. Let $R$ be the radical of $G$ and denote by $$\pi \colon G \to S = G/R$$ the natural projection. By Proposition \ref{hered} the homomorphisms $\pi \circ i \colon \Gamma \to S$ and $\pi \circ i \circ \theta \colon \Gamma \to S$ induce two embeddings of $\Gamma/(\Gamma \cap R)$ as cocompact lattices in $S$. By Theorem \ref{finconn} the number of connected components of the topological space $$\mathcal{R}(\Gamma/(\Gamma \cap R), S)$$ is finite. Therefore we can take a subgroup $H \subset \Aut(\Gamma)$ of finite index such that for every $\theta \in H$ the embeddings $\pi(\Gamma)$ and $\pi \circ \theta(\Gamma)$ lie in the same connected component of $\mathcal{R}(\Gamma/(\Gamma \cap R), S)$. Then by Theorem ~\ref{weil} the embeddings $\pi(\Gamma)$ and $\pi \circ \theta(\Gamma)$ are conjugate by an element of $S$, that is, $$\pi(\Gamma) = s\cdot \pi \circ \theta(\Gamma)\cdot s^{-1}$$ for some $s \in S$. Suppose that there exist elements $\theta \in H$ and $\gamma \in \Gamma \cap R$ such that $$\theta(\gamma) \notin \Gamma \cap R = \mathrm{Ker}(\pi|_{\Gamma}).$$ Then we have $$e = \pi(\gamma) = s\cdot\pi \circ \theta(\gamma) \cdot s^{-1} \neq e,$$ where $e$ denotes the identity element in $S$. This is a contradiction, so the subgroup $\Gamma \cap R \subset \Gamma$ is preserved by $H$. Consider the image of $H$ under the projection $\Aut(\Gamma) \to \mathrm{Out}(\Gamma)$. It is a subgroup of finite index in $\mathrm{Out}(\Gamma)$ preserving the subgroup $\Gamma \cap R$. Hence, $\mathrm{Out}(\Gamma; \Gamma \cap R)$ has finite index in $\mathrm{Out}(\Gamma)$.
\end{proof}

\section{Main results}

This section is devoted to the proofs of our main results. First, we prove Theorem \ref{main1}.

\begin{proof}[Proof of Theorem \ref{main1}] Consider the Levi--Mal'cev decomposition $$1 \to R \to G \to S \to 1,$$ where $R$ is the radical of $G$ and $S = G/R$ is semisimple. Then there is an induced exact sequence of discrete groups $$1 \to \Gamma \cap R \to \Gamma \to \Gamma/(\Gamma \cap R) \to 1,$$ where by Proposition \ref{hered} the group $\Gamma \cap R$ is a cocompact lattice in $R$ and $\Gamma/(\Gamma \cap R)$ is a cocompact lattice in $S$.

Let us show that the group $\mathrm{Out}(\Gamma; \Gamma \cap R)$ has bounded finite subgroups. We denote by $$B \colon \Aut(\Gamma; \Gamma \cap R) \to \Aut(\Gamma/\Gamma \cap R)$$ the natural homomorphism. By Proposition \ref{malfait} there is an exact sequence $$1 \to \Lambda_B \to \mathrm{Out}(\Gamma; \Gamma \cap R) \to \frac{\mathrm{Im}(B)}{\mathrm{Inn}(\Gamma/(\Gamma \cap R))} \to 1,$$
where the quotient group $$\mathrm{Im}(B)/\mathrm{Inn}(\Gamma/(\Gamma \cap R))$$ embeds to $\mathrm{Out}(\Gamma/(\Gamma \cap R))$. The latter group is finite by Theorem \ref{outgamma}. Thus it suffices to prove that $\Lambda_B$ has bounded finite subgroups. To do this, we consider the exact sequence \eqref{malfait2} from Theorem ~\ref{malfait}: $$1 \to \overline{H}_{\psi}^1(\Gamma/(\Gamma \cap R), Z(\Gamma \cap R)) \to \Lambda_B \to \Xi \to 1.$$ Since $\Gamma/(\Gamma \cap R) \subset S$ and $\Gamma \cap R \subset R$ are cocompact lattices in Lie groups, they are finitely generated, according to Theorem \ref{fingen}. Moreover, the center $Z(\Gamma \cap R)$ is a finitely generated abelian group, since the group $\Gamma \cap R$ is polycyclic by Proposition \ref{subquot}. Therefore by \cite[Exercises VIII.4.1 and VIII.5.1]{Bro82} the first cohomology group $$H_{\psi}^1(\Gamma/(\Gamma \cap R), Z(\Gamma \cap R))$$ is a finitely generated abelian group. Therefore the group $\overline{H}^1_{\psi}(\Gamma/(\Gamma \cap R), Z(\Gamma \cap R))$ is also finitely generated abelian, in particular (see Example \ref{abelianbfs}), it has bounded finite subgroups. On the other hand, by Theorem \ref{malfait} the quotient group $\Xi$ is a subgroup of $$\Upsilon = \mathrm{Out}(\Gamma \cap R)/\psi(Z(\Gamma/(\Gamma \cap R))),$$ where $\psi \colon \Gamma/(\Gamma \cap R) \to \mathrm{Out}(\Gamma \cap R)$ is the natural homomorphism. By Corollary \ref{outerpoly} the group $\mathrm{Out}(\Gamma \cap R)$ has bounded finite subgroups. Since $\Gamma/(\Gamma \cap R)$ is a lattice in the complex semisimple Lie group $S$, the group $Z(\Gamma/(\Gamma \cap R))$ is finite by Proposition \ref{zsemis}. Therefore by Remark \ref{quotbfs} the group $\Upsilon$ has bounded finite subgroups, hence the same property holds for the groups $\Xi, \Lambda_B$ and $\mathrm{Out}(\Gamma; \Gamma \cap R)$. 

By Proposition ~\ref{finindex} the subgroup $\mathrm{Out}(\Gamma; \Gamma \cap R)$ is of finite index in $\mathrm{Out}(\Gamma)$. Thus by Lemma \ref{bfsindex} the group $\mathrm{Out}(\Gamma)$ has bounded finite subgroups as well. 
\end{proof}

\begin{remark} It would be interesting to know if the conclusion of Theorem \ref{main1} holds for more general classes of groups, for example, for finitely presented subgroups of $\mathrm{GL}_n(\mathbb{C})$. It is known that the automorphism groups of finitely generated linear (over $\mathbb{C}$) groups are virtually torsion-free \cite[Corollary (5.2)]{BL83}, in particular, they have bounded finite subgroups. For the free group $F_r$ on $r$ generators the group $\mathrm{Out}(F_r)$ is virtually torsion-free \cite[Corollary (5.5)]{BL83}. Moreover, the orders of finite subgroups of $\mathrm{Out}(F_n)$ are bounded by 12 for $r = 2$ and by $2^rr!$ for $r \geqslant 3$ (see \cite{WZ94}).

On the other hand, there exist finitely generated subgroups $\Gamma \subset \mathrm{GL}_n(\mathbb{Z})$ such that the groups $\mathrm{Out}(\Gamma)$ have unbounded finite subgroups. Examples of such groups can be obtained by a version of the Rips construction from \cite{HW08}.
\end{remark}

Now we can prove Theorem \ref{main}.

\begin{proof}[Proof of Theorem \ref{main}] By Theorem \ref{wang} the manifold $X$ is isomorphic to the quotient $G/\Gamma$ of a connected complex Lie group $G$ by a cocompact lattice $\Gamma$. Replacing $G$ with its universal cover (see Remark \ref{univcover}), we may assume $G$ to be simply connected. Then by Proposition \ref{outer} we have an embedding $$\Aut(X)/\Aut^0(X) \subset \mathrm{Out}(\Gamma).$$ The group $\mathrm{Out}(\Gamma)$ has bounded finite subgroups by Theorem \ref{main1}, hence, the same is true for the group of connected components $\Aut(X)/\Aut^0(X)$. Since $X$ is compact, the connected component $\Aut^0(X)$ is a complex Lie group (by e. g. \cite[Theorem on p. 144]{Akh95} or by assertion (2) of Theorem \ref{auto}), so by Theorem \ref{liejor} it is Jordan. Thus we can apply Proposition \ref{Extensions} to the exact sequence $$1 \to \Aut^0(X) \to \Aut(X) \to \Aut(X)/\Aut^0(X) \to 1$$ and obtain that the group $\Aut(X)$ is Jordan.
\end{proof}

\begin{remark} Let $X$ be a compact parallelizable manifold. Then the group of bimeromorphic automorphisms $\mathrm{Bim}(X)$ coincides with $\Aut(X)$ (see \cite[Proposition 5.5.1]{Win98}).
\end{remark}

\begin{question} We expect that the methods used in the proofs of Theorems \ref{main} and \ref{main1} can be strenghened to establish other properties of the groups $\Aut(X)/\Aut^0(X)$ and $\mathrm{Out}(\Gamma)$. For instance, it is not clear if the group $\mathrm{Out}(\Gamma)$ is linear over $\mathbb{Z}$ in the general case.
\end{question}

\begin{question} Another intriguing problem is to find an effective upper bound for the Jordan constant of the automorphism group of a compact complex parallelizable manifold $X \simeq G/\Gamma$ in terms of invariants of $G$ and $\Gamma$.
\end{question}

\flushleft{National Research University Higher School of Economics, Russian Federation \\
Laboratory of Mirror Symmetry, NRU HSE\\
6 Usacheva str., Moscow, Russia, 119048 \\
and\\
Steklov Mathematical Institute of Russian Academy of Sciences, Moscow, Russia\\
8 Gubkina St., Moscow, 119991, Russia\\}
\email{golota.g.a.s@mail.ru, agolota@hse.ru}

\end{document}